\documentclass[10pt]{article}
\usepackage[OT2,T1]{fontenc}
\usepackage{amssymb,amsfonts,amsmath,amsthm,amscd}
\usepackage{relsize}
\usepackage{authblk}
\usepackage[margin=2.8cm]{geometry}
\usepackage{bigstrut}
\usepackage{mathtools}
\usepackage{titlefoot}
\usepackage{cite}
\usepackage[inline]{enumitem}
\usepackage[hidelinks]{hyperref}
\usepackage{orcidlink}

\DeclareMathOperator{\Aut}{Aut}

\DeclareMathOperator{\diag}{diag}
\DeclareMathOperator{\ddim}{ddim}
\DeclareMathOperator{\dind}{dind}

\DeclareMathOperator{\LL}{span}

\def\RR{\mathbb{R}}

\def\ZZ{\mathcal{Z}}

\def\t{\mathsmaller{T}}

\def\hn{\mathfrak{h}_{2n+1}}
\def\Hn{H_{2n+1}}

\def\DD{\mathcal{D}}

\newtheorem{lemma}{{Lemma}}[section]
\newtheorem{theorem}{{Theorem}}[section]
\newtheorem{proposition}{{Proposition}}[section]

\theoremstyle{definition}
\newtheorem{remark}{{Remark}}[section]

\newtheorem{examplex}{Example}[section]
\newenvironment{example}
  {\pushQED{\qed}\examplex}
  {\popQED\endexamplex}

\newtheorem*{thank}{{\indent Acknowledgement}}

\title{{Sub-Riemannian LR geodesic flows on Heisenberg groups}}
\author{Milan Pavlovi\'{c}\orcidlink{0009-0000-6079-9115}, Tijana \v{S}ukilovi\'{c}\orcidlink{0000-0001-6371-3081}}
\affil{\small{University of Belgrade, Faculty of Mathematics, Belgrade, Serbia}}
\date{\today}

\begin{document}
\allowdisplaybreaks

\maketitle\unmarkedfntext{2020 \emph{Mathematics Subject Classification}. 53C17, 22E25, 37J35\\
\indent\phantom{K}\emph{Key words and phrases}: Heisenberg group, left-invariant metrics, sub-Riemannian metric, right-invariant distributions, {LR} systems, geodesic flow, completely integrable Hamiltonian system.}

\begin{abstract}
{In this paper, we study the normal geodesic flow associated with a sub-Riemannian structure corresponding to a left-invariant Riemannian metric and a right-invariant distribution (LR structure) on the $(2n+1)$-dimensional Heisenberg group. We investigate the integrability properties of the corresponding LR system and establish a distinction between the low-dimensional and higher-dimensional cases. In dimensions greater than $5$, we prove that the LR system is integrable in the non-commutative sense, while in dimensions $3$ and $5$ we obtain the  integrability in the commutative sense. We compare these results with the known integrability properties of the corresponding LL systems, associated with a left-invariant metric and a left-invariant distribution, highlighting similarities and differences between the two types of sub-Riemannian geodesic flows.}
\end{abstract}

\section*{Introduction}

Sub-Riemannian geometry is a generalization of Riemannian geometry, sometimes described as Riemannian geometry with restrictions that emphasize its anisotropy. Among sub-Riemannian manifolds, nilpotent groups form a very important class. They are tangent cones of sub-Riemannian manifolds, which are the closest analogue of tangent spaces in Riemannian geometry, and the Heisenberg group is the most important and remarkable example of sub-Riemannian geometry that is not Riemannian. Although it is a relatively new area of geometry, sub-Riemannian geometry is a fruitful field of research, thanks to its wide range of applications, including robotics, optimal control theory, mechanics, hypoelliptic operator theory, geometric group theory, polynomial growth groups, and differential geometry.
Although it shares some properties with Riemannian manifolds, sub-Riemannian geometry has its own peculiarities, e.g. not all geodesics are solutions of Hamilton's system, and its Hausdorff dimension is always larger than its topological dimension -- a property that reflects the anisotropy in sub-Riemannian structures.

In this paper, we study the integrability of the sub-Riemannian geodesic flow induced by the left-invariant Riemannian metric on the Heisenberg group. We are interested in two different classes of this problem: the first is associated with the left-invariant distribution, and the second one with the right-invariant distribution. Examples of integrable sub-Riemannian geodesic flows on Lie groups and homogeneous spaces can be found, for example, in~\cite{BZ, BB, Sac, Ju, Mont2002, BJ3, JSV}. When investigating non-holonomic systems on Lie groups, an important class of problems arises that are modelled with so-called LR systems -- structures with right-invariant {nonholonomic constraints} and left-invariant {Riemannian} metric. Such systems are studied, for example, in~\cite{VV1998,FedJov}.  {For instance, consider the motion of a rigid body around a fixed point under a non-holonomic constraint: the projection of the angular velocity onto a fixed vector in space is constant.}

{Motivated by these perspectives and by the results of~\cite{Taim1997}, in addition to the jointly left-invariant distribution, we study the sub-Riemannian geodesic flows on the $(2n+1)$-dimensional Heisenberg group associated with a left-invariant Riemannian metric and a right-invariant contact distribution, and prove their integrability.} The {reduction} of this system corresponds to the motion of a charged particle in a constant magnetic field {in $\RR^{2n}$}. As far as we know, this is the first result beyond dimension three (see e.g.~\cite{Taim1997, Maz, AB}) for the integrability of geodesic flows for sub-Riemannian {LR structures on Lie groups}.

{Although our main focus is on LR systems, we also consider the corresponding LL systems in order to fix the notation and to highlight the similarities and differences between the two types of systems, particularly with respect to their integrability properties.}

The paper is structured as follows. In Section~\ref{sec:prelim}, we recall some basic concepts of sub-Riemannian geometry. Section~\ref{sec:subriem} introduces the $(2n+1)$-dimensional Heisenberg group and its sub-Riemannian structures. {In Section~\ref{sec:geodfl}, we consider two classes of sub-Riemannian geodesic flows, corresponding to LL and LR systems. In Subsection~\ref{ssec:LL}, we consider the LL case, where the normal geodesic flow is an integrable Hamiltonian system in the non-commutative sense. Its reduction to $\RR^{4n}$ is equivalent to the Hamiltonian system describing the motion of a charged particle in a constant magnetic field, whose regular phase space is foliated by invariant tori, with trajectories that are either closed or dense on these tori. These statements are summarized in Theorem~\ref{thm:left} and are consequences of the more general results previously obtained in~\cite{Jov,MM2017,Pod2023,BiggsNagy_2016}.}
Finally, {in Subsection~\ref{ssec:RL},} we present the main result of the paper. {We consider the sub-Riemannian structure associated with the general left-invariant Riemannian metric and the right-invariant distribution defining the Hamiltonian LR system in Proposition~\ref{prop:LRdef}.}
Similarly to the LL case, we also consider the {reduced} system and conclude that it is a system corresponding to the motion of a charged particle in a constant magnetic field (see Proposition~\ref{prop:rightR}).
For the {standard} metric, the LR Hamiltonian system is integrable in the commutative sense if $n=1$ (the case considered by Taimanov~\cite{Taim1997}) or $n=2$ (see Example~\ref{ex:5}). If $n> 2$, we have integrability in the non-commutative sense, {admitting first integrals of degrees one, two, and three in the momenta} (see Theorem~\ref{thm:rightR} for the {reduced} system and Theorem~\ref{thm:right} for the original one).

\section{Preliminaries}\label{sec:prelim}

{In this section, we recall some well-known facts from sub-Riemannian geometry, together with some basic notions and properties concerning the integrability of Hamiltonian systems.} Most of them are widely known in the literature and can be found, for example, in~{\cite{Mont2002,Str1986,VIA1978}}.

A sub-Riemannian structure on a manifold $Q$ of dimension $n$ is a pair $(\mathcal{D}, \langle \cdot,\cdot \rangle_{sR})$, where $\mathcal{D}$ is a distribution of constant rank $m$ of the tangent bundle $TQ$, and $\langle \cdot,\cdot \rangle_{sR}(q)$ is smoothly varying on $q\in Q$, a positive definite  bilinear form acting as an inner product on this distribution.
The distributions of interest are those that satisfy the bracket generating property. The bracket generating property means that at each point $q\in Q$ vector fields of $\mathcal{D}_q$ together with all their Lie brackets span all of $T_qQ$.

The almost everywhere absolutely continuous curve $\gamma: [0,1]\rightarrow Q$ is admissible, or horizontal, if $\dot{\gamma}(t)\in \mathcal{D}_{\gamma(t)}$ for almost every $t\in[0,1]$. For admissible $\gamma$ we define its length as
\begin{align*}
 \ell(\gamma)=\int_0^1 \|\dot{\gamma}(t)\|dt,
\end{align*}
where $\|\dot{\gamma}\|=\sqrt{\langle \dot{\gamma}, \dot{\gamma}\rangle_{sR}}$. In this case, the sub-Riemannian or Carnot-Carath\'{e}odory distance $dist_{CC}(x,y)=\inf \ \ell (\gamma)$ is defined, and the infimum is taken over all admissible curves $\gamma$ connecting $x$ and $y$. If there is no such curve, the Carnot-Carath\'{e}odory distance is infinite. For distributions with bracket-generating property, the theorems of Chow~\cite{Chow1939} and Rashievskii~\cite{Ras1938} guarantee that there always exists an admissible curve connecting any two points. The result for contact distributions is even older and goes back to the work of Carath\'{e}odory~\cite{Car1909} on Carnot cycles, which explains the name of the distance.

Cometric is a section of a bundle $S^2(TQ)\subset TQ\otimes TQ$ of symmetric bilinear forms on a cotangent bundle of a manifold $Q$. This contravariant two-tensor has the rank of a distribution which is generally smaller than the dimension of the manifold. To obtain a cometric, we introduce a linear map $\beta: T^{*}Q\rightarrow TQ$ that satisfies the following two conditions:
\begin{enumerate*}[label=(\roman*)]
 \item\label{it:1} ${\rm Im}( \beta_q)=\mathcal{D}_q$,
 \item\label{it:2} $\lambda(v)=\langle \beta_q(\lambda),v \rangle_{sR}$ for $v\in\mathcal{D}_q$ and $\lambda\in T^*Q$.
\end{enumerate*}
The mapping $\beta$ induces the bilinear form $(\cdot,\cdot)_q:T_q^*Q\otimes T_q^*Q\rightarrow \mathbb{R}$. More precisely, for $q\in Q$ and $p,\lambda \in T^{*}Q$ we have $(\lambda,p)_q=\lambda(\beta_q(p))$. The mapping $\beta_q$ is symmetric, i.e. it coincides with its dual mapping $\beta_q^*$, and cometric is symmetric and non-negative definite.

Every cometric of constant rank whose corresponding linear map $\beta$ satisfies the conditions~\ref{it:1} and~\ref{it:2} defines a sub-Riemannian structure, and conversely, for every sub-Riemannian structure, there is a unique $\beta$ that satisfies these conditions.

For the cometric $(\cdot,\cdot):T^*Q\times T^*Q\rightarrow \mathbb{R}$, the sub-Riemannian Hamiltonian or the kinetic energy is defined on $T^*Q$ as $H(q,\lambda)=(\lambda,\lambda)_q$.
Since the image of ${ \beta}$ is $\mathcal{D}$, for admissible curve $\gamma(t)$ there is a covector $\lambda(t)\in T^*_{\gamma(t)}Q$ such that $\dot{\gamma}={ \beta}_{\gamma(t)}(\lambda(t))$. We call $(\gamma(t),\lambda(t))$ cotangent lift and we have:
\begin{align*}
 H(q,\lambda)=\frac{1}{2}(\lambda,\lambda)_q=\frac{1}{2}\lambda({ \beta}_q(\lambda))=\frac{1}{2}\langle { \beta}_q(\lambda), { \beta}_q (\lambda)\rangle_{sR}=\frac{1}{2}\|\dot{\gamma}\|^2 .
\end{align*}
The preceding equations justify the notion of kinetic energy for the Hamiltonian and further justify the use of the Cauchy-Schwarz inequality for the length function for $\gamma : [0,1]\rightarrow Q$ as
\begin{align*}
 \ell (\gamma)\leq \sqrt{ \int \| \dot{\gamma}\|^2dt}= \sqrt{2 E(\gamma)},
\end{align*}
which shows that minimizing the curve length is equal to minimizing the energy. Equality is achieved for $\gamma$, parametrized by the length.

For the pairing $\lambda(X(q))$, $q\in Q$ between the fixed vector field $X$ and the covector $\lambda$ we write $P_X(q,\lambda)$. The function $P_X$ is referred to as momentum function.
For vector fields spanning distribution, written in coordinates as $X_a=\sum_{j} X_a^j(q)\frac{\partial}{\partial x_j}$, the momentum functions are given by
{$P_{X_a}=\sum_{j} X_a^j(q)\lambda_j$}. Let $g_{ab}(q)$ be the matrix of inner products of vectors $X_a$ and $g^{ab}(q)$ its inverse matrix. Then $g^{ab}(q)$ is an $m\times m$ matrix-valued function (where $m$ is the distribution rank) defined in some open set of $Q$.
It follows that Hamiltonian function can be written as
\begin{align*}
 H(q,\lambda)=\frac{1}{2}\sum_{a,b} g^{ab}(q)P_{X_a}(q,\lambda)P_{X_b}(q,\lambda),
\end{align*}
and the corresponding Hamilton's equations on the cotangent bundle are
\begin{align}\label{eq:HamEq}
 {\dot{q}}_i=\frac{\partial H}{\partial \lambda_i},\qquad\dot{\lambda}_i=-\frac{\partial H}{\partial {q}_i}.
\end{align}
Projections of solutions of Hamiltonian system are locally length-minimizing curves called normal geodesics. The following important theorem can be proved using the calculus of variations (see~\cite{Taim1997,Str1986}) or using Hamilton-Jacobi theory~\cite{Mont2002}.
\begin{theorem}[On a Hamiltonian structure for normal geodesic flow]\label{thm:HS}
  The projections of trajectories of the Hamiltonian flow on {$T^* Q$} with the Hamiltonian function $H(x,\lambda)=\frac{1}{2}g^{ij}\lambda_i\lambda_j$ are exactly the naturally-parametrized normal geodesics of the Carnot-Carath\'{e}dodory metric $g^{ij}$.
\end{theorem}

 In contrast to Riemannian geometry, in sub-Riemannian geometry we have locally minimizing curves that are not solutions of Hamilton's system. We call this type of geodesic an abnormal geodesic. In this paper, they are not of interest because for contact distributions, all geodesics are solutions of the Hamiltonian system~{\cite[Proposition 4.38]{ABB}.}

The Poisson bracket of two smooth functions is defined by using the Poisson tensor $\Lambda$ as usual $\{f, g\}\vert_x = {\Lambda_x}(df(x), dg(x))$, resulting in the Lie algebra structure of $C^\infty(S)$, $S=T^*Q$. In canonical coordinates $(q,{\lambda})$ the Poisson bracket is given by:
\begin{align*}
 \{f,g\}=\sum_i \frac{\partial f}{\partial q_i}\frac{\partial g}{\partial {\lambda}_i}- \frac{\partial f}{\partial {\lambda}_i}\frac{\partial g}{\partial q_i}.
\end{align*}

Let $f$ and $g$ be the first integrals of the Hamiltonian equations~\eqref{eq:HamEq} with the Hamiltonian $H$, i.e. $\{f,H\}=\{g,H\}=0$. Because of the Jacobi identity, $\{f,g\}$ is also the first integral. We can therefore consider the Lie algebra $\mathcal{F}$ of the first integrals.
Let $F_x = \{df \vert\, f \in \mathcal{F}\}$ be a subspace of $T^*_xS$, spanned by differentials of the functions from $\mathcal{F}$ at $x \in S$. It is assumed that the dimensions of $F_x$ and $\dim \ker\Lambda_x\vert F_x$ are constant on an open dense set $U$ of $S$. The corresponding dimensions are denoted by $\ddim \mathcal{F}$ (differential dimension of $\mathcal{F}$) and $\dind\mathcal{F}$
(differential index of $\mathcal{F}$).

$\mathcal{F}$ is a complete algebra if
\begin{align}\label{eq:complA}
 \ddim F + \dind F = \dim S.
\end{align}
{The Hamiltonian system~\eqref{eq:HamEq} is called \emph{completely integrable in the non-commutative sense} if it possesses a complete algebra $\mathcal{F}$ of first integrals.
For a Hamiltonian system on a $2n$-dimensional phase space, $\ddim\mathcal{F}$ is the number of functionally independent first integrals in $\mathcal{F}$. When $\ddim\mathcal{F}>n$, the system is superintegrable. In particular, a complete algebra with $\ddim\mathcal{F}=2n-1$ corresponds to maximal superintegrability.}

Mishchenko and Fomenko stated the conjecture that non-commutative integrability implies the Liouville integrability by means of an algebra of integrals that belong to the same functional class as the original one~\cite{MishFom}.
If $\mathcal{F}$ is a complete commutative algebra,
 $\ddim \mathcal{F}=\dind \mathcal{F}=\frac{1}{2}\dim S$,
we have the usual Liouville integrability.

\section{Sub-Riemannian structures on $\Hn$}\label{sec:subriem}

The Heisenberg group is a two-step nilpotent Lie group $\Hn$  defined on the base manifold $\RR ^{2n}\oplus \RR$ by multiplication
\begin{align}\label{eq:mult}
(u, \zeta )\cdot (v, \chi):= (u + v ,\, \zeta + \chi + \omega (u,v)).
\end{align}
Here $\omega$ denotes the standard symplectic form in the vector space $\RR ^{2n}$, which can be written in the following form
$\omega (u,v) = u^\t Jv$, $u, v \in \RR ^{2n}$, where
\begin{align}
	\label{eq:J}
J = J_{2n} =
\begin{pmatrix}
	0 & -E\\
	E & 0
\end{pmatrix}
\end{align}
and $E$ is the identity matrix of dimension $n\times n$.
This group can also be viewed as the group of matrices of the form:
\begin{align*}
  \begin{pmatrix}
    1 & x^\t & z\\
    0 & 1 & y\\
    0 & 0 & 1
  \end{pmatrix},\quad x,y\in \RR ^{n}, z\in\RR,
\end{align*}
with the standard matrix multiplication. This representation of the Heisenberg group is more convenient and will be used in the following.

The corresponding Lie algebra
$\hn = \RR ^{2n}\oplus \RR = \RR ^{2n}\oplus \ZZ = \{ (a, \alpha) \, | \, a\in \RR^{2n}, \alpha \in \RR\}$
is given by the following commutator equation
\begin{align}
	\label{eq:hnkom}
	[(a, \alpha ),(b, \beta)] = (0, \omega (a,b)).
\end{align}

Note that $\ZZ = \RR \langle\xi\rangle$, $\xi=(0,1)$, is one-dimensional centre and one-dimensional commutator subalgebra of $\hn$.
If  $e_1, \dots , e_n, f_1, \dots , f_n$ represents the standard basis of $\RR ^{2n}$, then non-zero commutators~\eqref{eq:hnkom} of $\hn$ are $[e_i, f_i] = \xi$, $i = 1, \dots , n$.

\begin{lemma}[\!\!\cite{Vuk2015}]
	Group of automorphisms of Heisenberg algebra $\hn$ denoted by $\Aut (\hn )$ can be described as  semi-direct product of symplectic group $Sp(2n, \RR),$ subgroup of translations isomorphic to $\RR ^{2n}$ and $1$-dimensional ideal, i.e.,
    \begin{align*}
        \Aut (\hn ) = \left\{\left.\begin{pmatrix}
            F & 0\\
            v^T & \alpha
        \end{pmatrix}\right| \alpha J=F^TJF, v\in\RR^{2n}, \alpha\neq 0\right\}.
    \end{align*}
\end{lemma}

\begin{theorem}[\!\!\cite{Vuk2015}]\label{thm:R}
Any left-invariant Riemannian metric on $\Hn$, up to the automorphism of Heisenberg algebra, is represented in the left-invariant basis by the block matrices
\begin{align}\label{eq:Riem}
  &\begin{pmatrix}
    D(\sigma) & 0& 0\\
    0 & D(\sigma) & 0\\
    0 & 0 & \tau
  \end{pmatrix},\quad \tau>0,
\intertext{where} &
\label{eq:Dsigma}
D(\sigma)=\diag(\sigma_1,\sigma_2,\ldots,\sigma_{n-1},\sigma_n), \quad \sigma_1\geq\sigma_2\geq\ldots\geq\sigma_{n-1}\geq\sigma_n=1.
\end{align}
\end{theorem}

\begin{theorem}[\!\!\cite{BiggsNagy_2013}]\label{thm:SR}
  Any left-invariant sub-Riemannian structure $(\DD, g)$ on $\Hn$ is isometric to exactly one of the structures $(\bar\DD, g^\sigma)$ given by:
 \begin{align}\label{ex:subriemann}
  \begin{cases}
   \bar\DD(e) = {\LL \{e_1, \dots , e_n, f_1, \dots , f_n\},}\\
   g^\sigma  = \diag(\sigma_1,\sigma_2,\ldots,\sigma_{n-1},1,\sigma_1,\sigma_2,\ldots,\sigma_{n-1},1),
   \end{cases}
 \end{align}
 with $\sigma_1\geq\sigma_2\geq\ldots\geq\sigma_{n-1}\geq1$.
\end{theorem}

\section{Geodesic flows of Carnot-Carath\'{e}odory metrics}\label{sec:geodfl}

In the following, we analyse two different classes of normal geodesic flows corresponding to the left-invariant metric on the $(2n+1)$-dimensional Heisenberg group as integrable Hamiltonian systems. The first class corresponds to the left-invariant distribution, while the second corresponds to the right-invariant one. The results for the LL systems are already known in a more general context and are only presented here to provide a self-contained presentation (see Theorem~\ref{thm:left}). The most important result of this section pertains to the LR systems (see Theorem~\ref{thm:rightR} and Theorem~\ref{thm:right}).

\subsection{Integrability of geodesic flow of the left-invariant metric on left-invariant distribution}\label{ssec:LL}

Consider the coordinates $(x_1,\ldots,x_n,y_1,\ldots,y_n,z)=(x,y,z)$.

The group $\Hn$ acts on itself by left translations $L_g$: $L_g(h)=gh$. Denote by $\DD_0$ the linear space spanned by $\frac{\partial}{\partial x_1}, \dots , \frac{\partial}{\partial x_n}, \frac{\partial}{\partial y_1}, \dots , \frac{\partial}{\partial y_n}$ at the identity. The left-invariant distribution generated by $\DD_0$ consists of $(2n)$-planes $\DD_L=L_{g*}\DD_0$, where:
\begin{align}\label{eq:lfields}
L_{g*}(\frac{\partial}{\partial x_k})=\frac{\partial}{\partial x_k},\quad L_{g*}(\frac{\partial}{\partial y_k})=\frac{\partial}{\partial y_k} + x_k \frac{\partial}{\partial z}\quad\text{and}\quad L_{g*}(\frac{\partial}{\partial z})=\frac{\partial}{\partial z}.
\end{align}
We note that $\Hn$ is a contact manifold and $\DD_L$ is a contact distribution.

The left-invariant Riemannian metric~\eqref{eq:Riem} in coordinates $(x,y,z)$ has the form:
\begin{align}\label{eq:leftR}
[g_{ij}(x,y,z)]=\begin{pmatrix}
    D(\sigma) & 0 & 0\\
    0 & \tau xx^\t+D(\sigma) & -\tau x\\
    0 & -\tau x^\t & \tau
  \end{pmatrix},
\end{align}
where $D(\sigma)$ is given by~\eqref{eq:Dsigma} and $\tau>0$.
The geodesic flow of the sub-Riemannian metric corresponding to the Riemannian metric~\eqref{eq:leftR} and the distribution $\DD_L$ is by Theorem~\ref{thm:HS} Hamiltonian {flow} on $T^*\Hn$ generated by the Hamiltonian function:
\begin{align}\label{eq:HamL}
  H(q,\lambda)=\frac{1}{2}\sum_{k=1}^n \frac{1}{\sigma_k}[\lambda_k^2+(\lambda_{k+n}+x_k\lambda_{2n+1})^2],
\end{align}
where $q=(x_1,\ldots,x_n,y_1,\ldots,y_n,z)$. The Hamilton equations have the form:
\begin{align}\label{eq:HamsysL}
  \dot{x}_k&=\{x_k, H\} =\frac{1}{\sigma_k}\lambda_k, & \dot{\lambda}_k &=\{\lambda_k, H\}=-\dot{y}_{k}\lambda_{2n+1},\notag\\
  \dot{y}_k&=\{y_k, H\}=\frac{1}{\sigma_k}(\lambda_{k+n}+x_k\lambda_{2n+1}), & \dot{\lambda}_{k+n}&=\{\lambda_{k+n}, H\}= 0,\\
  \dot{z} &=\{z, H\}= \sum_{j=0}^n x_j \dot{y}_{j} & \dot{\lambda}_{2n+1}&=\{\lambda_{2n+1}, H\}= 0,\notag
\end{align}
for $k=1,\ldots,n$.

Since $z$ is a cyclic coordinate, we can restrict the flow to the level set $\lambda_{2n+1}=C$. In this way, we {obtained the reduced} system on the $4n$-dimensional symplectic manifold, which is diffeomorphic to the cotangent bundle of the hyperplane $\RR^{2n}$ with coordinates $(x_1,\ldots,x_n,y_1,\ldots,y_n, \lambda_1,\ldots,\lambda_{2n})$, but with a different Poisson structure.

{The following theorem summarizes the complete integrability of the LL system and of its reduction to the level set $\lambda_{2n+1}=C\neq 0$.}

\begin{theorem}\label{thm:left}
\begin{enumerate}[wide, topsep=0pt, itemsep=10pt,  labelindent=0pt, label = (\alph*)]
  \item  The geodesic flow corresponding to the left-invariant Riemannian metric~\eqref{eq:leftR} on $\Hn$ and the left-invariant distribution $\DD_L$, is a Hamiltonian system on $T^*\Hn$ with the Hamiltonian function~\eqref{eq:HamL}  is completely integrable in the non-commutative sense via first integrals:
  \begin{align}\label{eq:intLL}
&I_0 =\lambda_{2n+1},\  I_k = \lambda_{k+n},\  I_{k+n}=\lambda_k+y_{k}\lambda_{2n+1},\  I_{2n+k}=\lambda_{k}^2+(\lambda_{k+n}+x_k\lambda_{2n+1})^2,\  k=1,\ldots, n.
\end{align}
These integrals are functionally independent almost everywhere, {and the integrals $I_0$ and $I_{2n+k}$ commute with all the other integrals.}
Moreover, the subset $\{\lambda_{2n+1}\neq 0\}$ of its phase space is foliated by general helices given by:
  \begin{align}\label{eq:traj}
    x_k(t) &= a_k \cos(\tau_k t) + b_k \sin(\tau_k t) + c_k,\quad
    y_k(t) = -b_k \cos(\tau_k t) + a_k \sin(\tau_k t) + d_k,\\
    z(t)&=c_0{+}\sum_{k=1}^n\left(\frac{1}{2} \tau _k t
   \left(a_k^2+b_k^2\right){+}\frac{1}{4}\left(a_k^2-b_k^2\right) \sin
   \left(2\tau _k t\right){-}\frac{1}{2} a_k b_k \cos \left(2\tau_k t\right){+} a_k c_k \sin \left(\tau _k t\right)-b_k c_k \cos \left(\tau_k t\right)\right), \notag
\end{align}
with $a_k,b_k,c_0,c_k,d_k\in\RR$ and $\tau_k=\dfrac{\lambda_{2n+1}}{\sigma_k}$, $k=1,\ldots,n$.

  \item The {reduced system} obtained by restricting the geodesic flow to the level set {$\lambda_{2n+1}=C\neq 0$} and {projecting this restriction of the flow} to the {plane $\RR^{4n}(x,y,\lambda)$} is equivalent to the Hamiltonian system describing the motion of a charged particle in the constant magnetic field $F=-C \sum_{i=1}^n \frac{1}{\sigma_i}\, dx_i\wedge dy_i$. {This system is integrable in non-commutative sense: on its $4n$-dimensional phase space, there exist $3n$ functionally independent first integrals, $n$ of which are in involution with all $3n$ integrals. The trajectories~\eqref{eq:traj} are closed or dense on a torus $T^n$ (depending on a rational dependency of $C$ and $\sigma_k$ constants from the metric).}
    \end{enumerate}
\end{theorem}

{
\begin{remark}
Note that LL system~\eqref{eq:HamsysL} is Liouville integrable and, in fact, superintegrable for every $n$. Indeed, the $2n+1$ first integrals $I_0$, $I_k$ and $I_{2n+k}$ from~\eqref{eq:intLL} are pairwise in involution.
Moreover, $H=\frac{1}{2}\sum_{k=1}^n\frac{1}{\sigma_k} I_{2n+k}$ and $\{I_k, I_{k+n}\}=I_0$, for all $k=1,\ldots,n$. The reduced system is likewise both Liouville integrable and superintegrable, and for $n=1$, it is even maximally superintegrable.
\end{remark}
}

 \begin{remark}
  Note that the vector field $X$ is a Killing vector field if and only if $\{H, h_X\}=0$, where $h_X(\lambda)=\langle\lambda, X(q)\rangle$ (see~\cite[Lemma 5.15]{ABB}). The concept of Killing vectors can be generalized to Killing tensors, although the geometric interpretation of metric symmetry does not transfer. The isomorphism that maps Killing tensors to homogeneous polynomial integrals is given by evaluation at the momenta $K\rightarrow I_K=K(\lambda,\ldots,\lambda)$.

  The linear first integrals $I_j$, $j=0,\ldots,2n$, from~\eqref{eq:intLL} are associated with the right-invariant vector fields~\eqref{eq:rfields}, i.e. the Killing vector fields corresponding to the left translations. The quadratic integrals $I_{2n+k}$ correspond to the Killing tensors $L_{g*}(\frac{\partial}{\partial x_k})\otimes L_{g*}(\frac{\partial}{\partial x_k})+L_{g*}(\frac{\partial}{\partial y_k})\otimes L_{g*}(\frac{\partial}{\partial y_k})$.
 \end{remark}

\subsection{Integrability of geodesic flow of the left-invariant metric on right-invariant distribution}\label{ssec:RL}

Now, let us fix the right-invariant distribution is $\DD_R=R_{g*}\DD_0$, where $R_g$ denotes the right translation and $R_{g*}:T_e\Hn\rightarrow T_g\Hn$:
\begin{align}\label{eq:rfields}
R_{g*}(\frac{\partial}{\partial x_k})=\frac{\partial}{\partial x_k}+ y_k \frac{\partial}{\partial z},\quad R_{g*}(\frac{\partial}{\partial y_k})=\frac{\partial}{\partial y_k}\quad\text{and}\quad R_{g*}(\frac{\partial}{\partial z})=\frac{\partial}{\partial z}.
\end{align}

{
In local coordinates $(x_1,\ldots,x_n,y_1,\ldots,y_n,z)=(x,y,z)$ the Hamiltonian function $H(x,y,z,\lambda)=\frac{1}{2}\sum_{i,j} g^{ij}\lambda_i\lambda_j$ can be written as follows:
\begin{align}\label{eq:HamR}
H=\frac{1}{2}&\left(\bar\lambda^\t (D(\sigma)^{-1}-\frac{\tau}{1{+}\tau\langle D(\sigma)^{-1}q, q\rangle}( D(\sigma)^{-1} Jq)( D(\sigma)^{-1} Jq)^\t) \bar\lambda\right.\notag\\
&\left.{+}2\lambda_{2n{+}1}\bar\lambda^\t D(\sigma)^{-1}(u{+}\frac{\tau\langle D(\sigma)^{-1} u, u\rangle}{1{+}\tau\langle D(\sigma)^{-1} q, q\rangle}Jq){+}\lambda_{2n{+}1}^2(\langle D(\sigma)^{-1} u, u\rangle-\frac{\tau\langle D(\sigma)^{-1} u, u\rangle^2}{1{+}\tau\langle D(\sigma)^{-1} q, q\rangle})\right),
\end{align}
where $(x,y,z)=(q,z)$, $u=(y,0)$, $\lambda=(\lambda_1,\ldots,\lambda_{2n},\lambda_{2n+1})=(\bar\lambda,\lambda_{2n+1})$, where $J$ is given by~\eqref{eq:J}, $D(\sigma)$ is given by~\eqref{eq:Dsigma}  and $\langle\cdot,\cdot\rangle$ represents the standard inner product on $\RR^{2n}$. Thus, the corresponding Hamilton equations are:
    \begin{gather}
\begin{aligned}\label{eq:HamReq}
    &\dot{q} =\{q, H\} =  (D(\sigma)^{-1}-\frac{\tau}{1+\tau\langle D(\sigma)^{-1} q, q\rangle}( D(\sigma)^{-1} Jq)( D(\sigma)^{-1} Jq)^\t)\bar{\lambda}\\&\phantom{\dot{q} =\{q, H\} =(}+\lambda_{2n+1}D(\sigma)^{-1}(u+\frac{\tau\langle D(\sigma)^{-1}u, u\rangle}{1+\tau\langle D(\sigma)^{-1}q, q\rangle}Jq),\\
    &\dot{z}=\{z, H\}=\bar\lambda^\t D(\sigma)^{-1}(u+\frac{\tau\langle D(\sigma)^{-1}u, u\rangle}{1+\tau\langle D(\sigma)^{-1}q, q\rangle}Jq)+\lambda_{2n+1}(\langle D(\sigma)^{-1}u, u\rangle-\frac{\tau\langle D(\sigma)^{-1}u, u\rangle^2}{1+\tau\langle D(\sigma)^{-1}q, q\rangle}),\\
    &\dot{\bar\lambda}=\{\bar\lambda, H\}=-D(\sigma)^{-1}\left(C_{q,\bar{\lambda}}(C_{q,\bar{\lambda}} q - J \bar{\lambda}) -\lambda_{2n+1}((2C_{q,\bar{\lambda}}-\lambda_{2n+1})Ju-v)\right),\\
    &\dot{\lambda}_{2n+1}=\{\lambda_{2n+1}, H\}=0,
\end{aligned}
\end{gather}
where $v=(0,\ldots,0,\lambda_1,\ldots,\lambda_n)$ and $C_{q,\bar{\lambda}}=\dfrac{\tau}{1+\tau\langle D(\sigma)^{-1}q,q\rangle}(\omega(D(\sigma)^{-1}q,\bar{\lambda})+\lambda_{2n+1}\langle D(\sigma)^{-1}u,u\rangle)$.

The preceding discussion establishes the result stated below.

\begin{proposition}\label{prop:LRdef}
    The geodesic flow of the sub-Riemannian metric induced by the general Riemannian metric~\eqref{eq:leftR} on the distribution $\DD_R$ is the Hamiltonian flow on $T^*\Hn$ with the Hamiltonian function~\eqref{eq:HamR}. The corresponding Hamiltonian system is given by~\eqref{eq:HamReq}.
\end{proposition}
}

 Similarly to the LL case, we can formulate the following proposition.
 \begin{proposition}\label{prop:rightR}
   {The reduced Hamiltonian system obtained by symplectic reduction of the LR system~\eqref{eq:HamReq} at the momentum level} $\lambda_{2n+1}=C$ is equivalent to the Hamiltonian system describing the motion of a charged particle in the constant magnetic field $F=C\sum_{i=1}^n dx_i\wedge dy_i$.
 \end{proposition}

\begin{proof}
{It is evident from~\eqref{eq:HamReq} that} the variable $z$ is cyclic. Hence, we can restrict the flow to the level set $\lambda_{2n+1} = C$ and project this restriction onto $\RR^{ 4n}$. In this way, we reduce the problem to the Hamiltonian system on the $4n$-dimensional symplectic manifold. Set:
\begin{align}\label{eq:l2p}
  p_k = \lambda_k + C y_k,\quad p_{k+n} = \lambda_{k+n}, \quad k=1,\ldots,n.
\end{align}
The Hamiltonian function~{\eqref{eq:HamR}} takes the form:
\begin{align}\label{eq:Hc}
  H_C(q, p)=\frac{1}{2}\left(\langle D(\sigma)^{-1}p, p\rangle-\frac{\tau}{1+\tau\langle D(\sigma)^{-1} q, q\rangle}\,\omega(D(\sigma)^{-1}q,p)^2\right).
\end{align}
The Poisson structure on the restricted manifold is given by:
\begin{align*}
  \{x_k,p_j\}&=\{y_k,p_{j+n}\}=\delta_{kj},\quad \{p_k,p_{j+n}\}=\delta_{kj}C,\\
   \{x_k, y_j\}&=\{x_k,p_{j+n}\}=\{y_k, p_j\}=\{p_j, p_k\}=\{p_{j+n}, p_{k+n}\}=0,
\end{align*}
for $k,j=1,\ldots,n$, and the Hamilton's system is:
\begin{gather}\label{eq:HameqR}
\begin{aligned}
    &\dot{q}=\{q, H_C\} = D(\sigma)^{-1}\left(p-\frac{\tau\omega(D(\sigma)^{-1}q,p)}{1+\tau\langle D(\sigma)^{-1}q, q\rangle} J q\right),\\
    &\dot{p}=\{p, H_C\}=-\left(C +\frac{\tau\omega(D(\sigma)^{-1}q,p)}{1+\tau\langle D(\sigma)^{-1} q, q\rangle}\right)JD(\sigma)^{-1}\left(p-\frac{\tau\omega(D(\sigma)^{-1}q,p)}{1+\tau\langle D(\sigma)^{-1}q, q\rangle}Jq\right).
    \end{aligned}
    \end{gather}
Note that this system describes the motion of a charged particle in the constant magnetic field $F=C\sum_{i=1}^n  dx_i\wedge dy_i$.
\end{proof}

In the sequel, let us consider the case of the standard metric, i.e., let us assume that $D(\sigma)$ is the identity matrix and $\tau=1$.

\begin{theorem}\label{thm:rightR}
 {The reduced system~\eqref{eq:HameqR}, with $D(\sigma)=E$ and $\tau=1$, is Liouville integrable for $n=1,2$. For $n\geq3$, it is integrable in the non-commutative sense and possesses $3n-2$ first integrals. In hyperspherical coordinates~\eqref{eq:hscoor}, these first integrals are given by~\eqref{eq:intRLred}--\eqref{eq:intRLredN}.}
\end{theorem}

\begin{proof}
We introduce the new coordinates $(r, \theta_1,\theta_2,{\ldots},\theta_{2n-1})$:
\begin{gather}
\begin{aligned}\label{eq:hscoor}
  x_{1} &= r \prod_{j=1}^{n-2}\cos(\theta_j)\cdot\cos(\theta_{n-1})\cdot\cos(\theta_{n}), & y_{1}&=r \prod_{j=1}^{n-2}\cos(\theta_j)\cdot\cos(\theta_{n-1})\cdot\sin(\theta_{n}),\\
 x_{k} &= r \prod_{j=1}^{n-k}\cos(\theta_j)\cdot\sin(\theta_{n-k+1})\cdot\cos(\theta_{n+k-1}), & y_{k}&=r \prod_{j=1}^{n-k}\cos(\theta_j)\cdot\sin(\theta_{n-k+1})\cdot\sin(\theta_{n+k-1}),
\end{aligned}
\end{gather}
for $k=2,\ldots,n$.  The Riemannian metric in these coordinates takes the form:
\begin{align*}
  dr^2 {+} {r^2} d\theta_1^2 {+} {r^2}\sum_{k=1}^{n-1}{\prod_{j=1}^{k-1}\cos^2(\theta_j)}\left(\cos^2(\theta_k){d\theta_{k+1}^2}{+}\sin^2(\theta_k){d\theta_{2n-k}^2}\right){+}r^4\sum_{k,j=1}^n \psi_k\psi_j d\theta_{n-1+k}d\theta_{n-1+j},
\end{align*}
with $\psi_1=\prod_{i=1}^{n-1}\cos^2(\theta_i)$ and $\psi_k=\prod_{i=1}^{n-k}\cos^2(\theta_i)\sin^2(\theta_{n-k+1})$, $k=2,\ldots,n$.

The associated coordinates $p_r$ and $p_{\theta_k}$, $k=1,\ldots,2n-1$, are:
\begin{align*}
  p_r = \frac{\langle q, p\rangle}{\sqrt{\langle q, q\rangle}},\quad
  p_{\theta_j}  = \begin{cases}
    \frac{1}{\alpha_{n+1-j}}\beta_{n+1-j}-\alpha_{n+1-j}\sum_{k=j+1}^n\beta_{n+1-k},& 1\leq j < n,\\
    x_{j+1-n} p_j - y_{j+1-n}p_{j+1-n},& n\leq j < 2n,
  \end{cases}
\end{align*}
where $\alpha_j=\frac{\sqrt{x_j^2+y_j^2}}{\sqrt{\langle q, q\rangle-(x_j^2+y_j^2)}}$ and $\beta_j = x_j p_j+y_j p_{n+j}$.

The Hamiltonian function~{\eqref{eq:Hc}} now becomes:
\begin{align*}
  \tilde H_C(r,\theta,p_r,p_\theta)=\frac{1}{2}\left(p_r^2 +\frac{1}{r^2}\sum_{k=1}^{n-1}\frac{p_{\theta_k}^2}{\prod_{j=1}^{k-1}\cos^2(\theta_{j})}
  +\frac{1}{r^2}\rho_{n-1} - \frac{1}{1+r^2}\left(\sum_{k=n}^{2n-1}p_{\theta_k}\right)^2\right),
\end{align*}
where $\rho_k=\cos^{-2}(\theta_{n-k})\rho_{k-1}+{\sin^{-2}(\theta_{n-k})}p_{\theta_{n+k}}^2$, $\rho_0=p_{\theta_n}^2$.
Hamilton's equations are too complicated to write down in full and will be given only in the special case where $n=2$ (see Example~\ref{ex:5}). Here, we only provide the relationships that are necessary to determine the integrals:
\begin{flalign*}
 && &\{p_{\theta_{n}},  \tilde H_C\}=-\beta_1\lambda_{2n+1}=-\frac{ C}{2}\{r^2\prod_{k=1}^{n-2}\cos^2\theta_k\cos^2\theta_{n-1},  \tilde H_C\},&\\
&&  &\{p_{\theta_{i+n-1}},  \tilde H_C\}=-\beta_i\lambda_{2n+1}=-\frac{ C}{2}\{r^2\prod_{k=1}^{n-i}\cos^2\theta_k\sin^2\theta_{n-i+1},  \tilde H_C\},& 1< i\leq n,\\
&& & \{p^2_{\theta_{n-1}}, \tilde H_C\}=\{(p_{\theta_n}+p_{\theta_{n+1}})^2-p_{\theta_n}^2\cos^{-2}\theta_{n-1}-p_{\theta_{n+1}}^2\sin^{-2}\theta_{n-1},\tilde H_C\},&\\
&& & \{(p_{\theta_n}\prod_{k=1}^{i-1}\cos^{-1}\theta_{n-k}\tan\theta_{n-i}-p_{\theta_{n+i}}\prod_{k=1}^{i-1}\cos\theta_{n-k}\cot\theta_{n-i})^2, \tilde H_C\}=&\\
&& &-\{(p_{\theta_{n-i}}\prod_{k=1}^{i-1}\cos\theta_{n-k}+\varrho_i\tan\theta_{n-i})^2\},& 1 < i < n,\\
\text{where} && &\varrho_i=p_{\theta_{n-i}}\sin\theta_{n-i}\prod_{k=1}^{i-1}\cos\theta_{n-k}+\cos^{-1}\theta_{n-i}\varrho_{i-1},\ \varrho_0=0.&
\end{flalign*}
We deduce that this reduced system has the following first integrals:
\begin{gather}
\begin{aligned}\label{eq:intRLred}
  \tilde{I}_0 &= \tilde H_C,\\
  \tilde{I}_1 &= p_{\theta_{n}}+\frac{ C}{2}r^2\prod_{k=1}^{n-2}\cos^2\theta_k\cos^2\theta_{n-1},\\
  \tilde{I}_i &=p_{\theta_{i+n-1}}+\frac{ C}{2}r^2\prod_{k=1}^{n-i}\cos^2\theta_k\sin^2\theta_{n-i+1},\quad 1<i\leq n,\\
  \tilde{I}_{n+1}&=p^2_{\theta_{n-1}}-(p_{\theta_n}+p_{\theta_{n+1}})^2+p_{\theta_n}^2\cos^{-2}\theta_{n-1}+p_{\theta_{n+1}}^2\sin^{-2}\theta_{n-1},\\
  \tilde{I}_{n+i}&=(p_{\theta_n}\prod_{k=1}^{i-1}\cos^{-1}\theta_{n-k}\tan\theta_{n-i}-p_{\theta_{n+i}}\prod_{k=1}^{i-1}\cos\theta_{n-k}\cot\theta_{n-i})^2\\
  &+  (p_{\theta_{n-i}}\prod_{k=1}^{i-1}\cos\theta_{n-k}+\varrho_i\tan\theta_{n-i})^2,\quad 1<i<n.
\end{aligned}
\end{gather}

{To show that the integrals are functionally independent, we consider the Jacobian with respect to the momentum variables. For the $n$ linear integrals, we have
\begin{align*}
    \frac{\partial \tilde{I}_i}{\partial p_r}=\frac{\partial \tilde{I}_i}{\partial p_{\theta_k}}=0,\quad \frac{\partial \tilde{I}_i}{\partial p_{\theta_{j+n-1}}}=\delta_{ij},\quad k=1,\ldots n-1,\ i, j=1,\ldots n.
\end{align*}
Hence, the first $n$ rows of the Jacobian contain an $n\times n$ identity block, corresponding to the momentum variables $ p_{\theta_n},\ldots, p_{\theta_{2n-1}}$. In particular, the first $n$ integrals are functionally independent.

We next consider the remaining momentum variables $p_r, p_{\theta_1},\ldots, p_{\theta_{n-1}}$.
For the first quadratic integral, we have
\begin{align*}
    \frac{\partial \tilde{I}_{n+1}}{\partial p_r}=\frac{\partial \tilde{I}_{n+1}}{\partial p_{\theta_k}}=0,\quad \frac{\partial \tilde{I}_{n+1}}{\partial p_{\theta_{n-1}}}=2p_{\theta_{n-1}},\quad k=1,\ldots n-2.
\end{align*}
For $i=2,\ldots,n-1$, the integral $\tilde{I}_{n+i}$ does not depend on $p_r$ or on $p_{\theta_{n-j}}$, for $j<i$, whereas
\begin{align*}
    \frac{\partial \tilde{I}_{n+i}}{\partial p_{\theta_{n-i}}}=2c_i(p_{\theta_{n-i}}c_i+\varrho_{i-1}\cos^{-1}\theta_{n-i}\tan\theta_{n-i}),\quad c_i=(1+\sin\theta_{n-i}\tan\theta_{n-i})\prod_{k=1}^{i-1}\cos\theta_{n-k}.
\end{align*}
Thus, with respect to the ordered variables $p_r, p_{\theta_1},\ldots, p_{\theta_{n-1}}$ the Jacobian of the quadratic integrals has a triangular structure.

Finally, for the Hamiltonian $\tilde{I}_0 = \tilde H_C$, we have $\dfrac{\partial \tilde{I}_0}{\partial p_r}=p_r$ while none of the other $2n-1$ integrals depends on $p_r$. Consequently, the full $2n\times 2n$ Jacobian contains a block-triangular minor with a determinant that is not identically zero.}

Note that $\{\tilde{I}_{n+i}, \tilde{I}_{n+j}\}\neq 0$, for every $1\leq i < j < n$.

For $n=1,2$, the algebra of integrals~\eqref{eq:intRLred} is complete and involutive, and $\ddim\mathcal{F}=\dind\mathcal{F}=2n$. {Hence, the case ($n=1,2$) is complete, and the corresponding reduced Hamiltonian system is Liouville integrable.}

{For $n> 2$, among all pairwise Poisson brackets $\{\tilde{I}_{n+i}, \tilde{I}_{n+j}\}$, $1\leq i < j < n$, we select the cubic integrals
\begin{align}\label{eq:intRLredN}
    \tilde{J}_{k} = \{\tilde{I}_{n+1}, \tilde{I}_{n+k}\},\quad 1 < k < n.
\end{align}
The formulas for these integrals are omitted because their form is too complicated.  However, these integrals form a maximal functionally independent subset of the above collection of brackets. Indeed, for every $1 < i < k < n$, the integral $\{\tilde{I}_{n+i}, \tilde{I}_{n+k}\}$ is functionally dependent on  $\tilde{I}_1,\ldots,\tilde{I}_n$, $\tilde{I}_{n+1}, I_{n+i}, I_{n+k}$, $\tilde{J}_{i}$ and $\tilde{J}_{k}$. Hence, with $\tilde{I}_j, \tilde{J}_k$, $1\leq j<2n$, $1<k<n$, set, no additional functionally independent integrals can be obtained from the brackets.

It remains to establish the functional independence of the integrals $\tilde{J}_{k}$.
For each $1 < k < n$, the monomial $p_{\theta_{n-1}}p_{\theta_{n+k}}^2$ appears uniquely in $\tilde{J}_{k}$. Moreover, none of the $\tilde{J}_{k}$ depends on $p_r$. Hence, the cubic integrals are functionally independent and, in particular, independent of the Hamiltonian $\tilde{I}_0 = \tilde H_C$. To establish their functional independence from the linear and quadratic integrals $\tilde{I}_j$, $1\leq j< 2n$, it suffices to evaluate the full Jacobian at a suitably chosen point. We set $\theta_j=\frac{\pi}{4}$ and $p_{\theta_j}=0$ for $j\neq n-1$. As established above, the first $n$ rows have the same block structure: an $n\times n$ identity block corresponding to $p_{\theta_n},\ldots, p_{\theta_{2n-1}}$, together with an $n\times n$ block of entries $c_{ij} r^{2-\delta_{1j}}$  corresponding to  $r, \theta_1,\ldots,\theta_{n-1}$. At the chosen point, the only nonzero blocks in the rows corresponding to  $\tilde{I}_{n+i}$, $1\leq i <n$, are the $(n-1)\times(n-1)$ blocks
\begin{align*}
    A=\left(\frac{\partial \tilde{I}_{n+i}}{\partial\theta_j}\right)=A_{ij} p_{\theta_{n-1}}^2\quad\text{and}\quad B=\left(\frac{\partial \tilde{I}_{n+i}}{\partial p_{\theta_j}}\right)=B_{ij} p_{\theta_{n-1}}.
\end{align*}
Both exhibit an anti-triangular structure, with $A_{ij}=B_{ij}=0$ for $i+j\leq n$. All anti-diagonal entries, except $A_{1,n-1}$,  are non-zero. Finally, the cubic integrals have the following non-vanishing derivatives:
\begin{align*}
    \frac{\partial\tilde{J}_{k}}{\partial\theta_j}&=\begin{cases}
        0, & j+k\leq n,\\
        c_{jk}p_{\theta_{n-1}}^3, & j+k> n,
    \end{cases}\ 1\leq j<n-1,\\
    \frac{\partial\tilde{J}_{k}}{\partial p_{\theta_{n-1}}}&=-3\cdot2^k p_{\theta_{n-1}}^2,\quad \frac{\partial\tilde{J}_{k}}{\partial p_{\theta_{j}}}=c_{kj}r^2 p_{\theta_{n-1}},\quad n\leq j < 2n.
\end{align*}
The above block structure of the Jacobian, in particular the triangular structure of the relevant blocks and their non-vanishing entries, shows that the Jacobian has maximal rank at the chosen point. Hence, the integrals are functionally independent.
}

The total number of first integrals is $\ddim\mathcal{F}=3n-2$, {with $I_0,\ldots,I_n$ in involution with all first integrals. For determining the differential index, it is sufficient to consider the sub-matrix associated with the non-involutive integrals:}
\begin{align*}
    \begin{pmatrix}
        \{\tilde{I}_{n+1}, \tilde{I}_{n+1}\} & \{\tilde{I}_{n+1}, \tilde{I}_{n+k}\} & \{\tilde{I}_{n+1}, \tilde{J}_{k}\}\\
        \{\tilde{I}_{n+j}, \tilde{I}_{n+1}\} &  \{\tilde{I}_{n+j}, \tilde{I}_{n+k}\} &  \{\tilde{I}_{n+j}, \tilde{J}_{k}\}\\
        \{\tilde{J}_{j}, \tilde{I}_{n+1}\} &  \{\tilde{J}_{j}, \tilde{I}_{n+k}\} &   \{\tilde{J}_{j}, \tilde{J}_{k}\} &
    \end{pmatrix}.
\end{align*}
It is skew-symmetric, so its rank must be an even number. Since it is formed from $2n-3$ functionally independent integrals, the co-rank is exactly one.
Hence, $\dind\mathcal{F}=n+2$ and from~\eqref{eq:complA} we deduce that {for $n>2$} the reduced system is integrable in the non-commutative sense.
\end{proof}

Let us illustrate this theorem with an example.

\begin{example}\label{ex:5}
 For example, let us consider the geodesic flow of the sub-Riemannian metric corresponding to the left-invariant metric on $H_5$ and the right-invariant distribution $\DD_R=R_{g*}\DD_0$. This geodesic flow constitutes the Hamiltonian system on $T^* H_5$ with the Hamiltonian function:
    \begin{align*}
        H(q,\lambda)&=\frac{1}{2\left(x_1^2{+}x_2^2{+}y_1^2{+}y_2^2{+}1\right)}
        \left(\left(x_1^2{+}x_2^2{+}y_2^2{+}1\right) \lambda _1^2{+}\left(x_1^2{+}x_2^2{+}y_1^2{+}1\right) \lambda _2^2{+}\left(x_2^2{+}y_1^2{+}y_2^2{+}1\right) \lambda _3^2 \right.\\
   &{+}\left(x_1^2{+}y_1^2{+}y_2^2{+}1\right) \lambda _4^2{+}\left(x_1^2{+}x_2^2{+}1\right)
   \left(y_1^2{+}y_2^2\right) \lambda _5^2{-}2 x_1 x_2 \lambda _3 \lambda _4{+}2
   \left(y_1^2{+}y_2^2\right) \left(x_1 \lambda _3{+}x_2 \lambda _4\right) \lambda_5\\
   &\left.{+}2 y_1 \left({-}y_2 \lambda _2{+}x_1
   \lambda _3{+}x_2 \lambda _4{+}\left(x_1^2{+}x_2^2{+}1\right) \lambda _5\right) \lambda_1{+}2 y_2 \lambda _2 \left(x_1 \lambda _3{+}x_2 \lambda
   _4{+}\left(x_1^2{+}x_2^2{+}1\right) \lambda _5\right)
        \right),
    \end{align*}
    where $q=(x_1,x_2,y_1,y_2,z)$ and $\lambda=(\lambda_1,\ldots,\lambda_5)$. The Hamiltonian system takes the form:
\begin{align*}
      \dot{x_1} = \{x_1, H \} &= \frac{(1{+}x_1^2{+}x_2^2{+}y_2^2)\lambda_1{+}y_1(-y_2\lambda_2{+}x_1\lambda_3{+}x_2\lambda_4{+}(1{+}x_1^2{+}x_2^2)\lambda_5)}{1{+}x_1^2{+}x_2^2{+}y_1^2{+}y_2^2},\\
      \dot{x_2} = \{x_2, H \} &= \frac{-y_1y_2\lambda_1(1{+}x_1^2{+}x_2^2)\lambda_2{+}y_1^2\lambda_2{+}y_2(x_1\lambda_3{+}x_2\lambda_4{+}(1{+}x_1^2{+}x_2^2)\lambda_5)}{1{+}x_1^2{+}x_2^2{+}y_1^2{+}y_2^2},\\
      \dot{y_1} = \{y_1, H \} &= \frac{(1{+}x_1^2{+}x_2^2{+}y_2^2)\lambda_3{+}x_1(y_1\lambda_1-x_2\lambda_4{+}y_1^2\lambda_5{+}y_2(\lambda_2{+}y_2\lambda_5)}{1{+}x_1^2{+}x_2^2{+}y_1^2{+}y_2^2},\\
      \dot{y_2} = \{y_2, H \} &= \frac{(1{+}x_1^2{+}x_2^2{+}y_2^2)\lambda_4{+}x_2(y_1\lambda_1-x_1\lambda_3{+}y_1^2\lambda_5{+}y_2(\lambda_2{+}y_2\lambda_5)}{1{+}x_1^2{+}x_2^2{+}y_1^2{+}y_2^2},\\
      \dot{z} = \{z, H \} &= \frac{(1{+}x_1^2{+}x_2^2)y_1\lambda_1{+}(1{+}x_1^2{+}x_2^2)y_2\lambda_2{+}(y_1^2{+}y_2^2)(x_1\lambda_3{+}x_2\lambda_4){+}(1{+}x_1^2{+}x^2)(y_1^2{+}y_2^2)\lambda_5}{1{+}x_1^2{+}x_2^2{+}y_1^2{+}y_2^2},\\
      \dot{\lambda_1} = \{\lambda_1, H \} &= - \frac{1}{(1{+}x_1^2{+}x_2^2{+}y_1^2{+}y_2^2)^2}(y_1\lambda_1-x_1\lambda_3-x_2\lambda_4{+}y_1^2\lambda_5{+}y_2(\lambda_2{+}y_2\lambda_5))\cdot\\
      &\phantom{=-}((1{+}x_2^2{+}y_1^2{+}y_2^2)\lambda_3{+}x_1(y_1\lambda_1-x_2\lambda_4{+}y_1^2\lambda_5{+}y_2(\lambda_2{+}y_2\lambda_5))),\\
      \dot{\lambda_2} = \{\lambda_2, H \} &= - \frac{1}{(1{+}x_1^2{+}x_2^2{+}y_1^2{+}y_2^2)^2}(y_1\lambda_1-x_1\lambda_3-x_2\lambda_4{+}y_1^2\lambda_5{+}y_2(\lambda_2{+}y_2\lambda_5))\cdot\\
      &\phantom{=-}((1{+}x_1^2{+}y_1^2{+}y_2^2)\lambda_4{+}x_2(y_1\lambda_1-x_1\lambda_3{+}y_1^2\lambda_5{+}y_2(\lambda_2{+}y_2\lambda_5))),\\
      \dot{\lambda_3} = \{\lambda_3, H \} &= \frac{1}{(1{+}x_1^2{+}x_2^2{+}y_1^2{+}y_2^2)^2}(y_1\lambda_1{+}x_4\lambda_2-x_1\lambda_3-x_2\lambda_4-(1{+}x_1^2{+}x_2)\lambda_5)\cdot\\
      &\phantom{=-}((1{+}x_1^2{+}x_2^2{+}y_2^2)\lambda_1{+}x_3(-y_2\lambda_2{+}x_1\lambda_3{+}x_2\lambda_4{+}(1{+}x_1^2{+}x_2^2)\lambda_5)),\\
      \dot{\lambda_4} = \{\lambda_4, H \} &= \frac{1}{(1{+}x_1^2{+}x_2^2{+}y_1^2{+}y_2^2)^2}(y_1\lambda_1{+}x_4\lambda_2-x_1\lambda_3-x_2\lambda_4-(1{+}x_1^2{+}x_2)\lambda_5)\cdot\\
      &\phantom{=-}(-y_1y_2\lambda_1{+}(1{+}x_1^2{+}x_2^2)\lambda_2{+}y_1^2\lambda_2{+}y_2(x_1\lambda_3{+}x_2\lambda_4{+}(1{+}x_1^2{+}x_2^2)\lambda_5)),\\
      \dot{\lambda_5}=\{ z,H\} &= 0.
    \end{align*}
    This flow has the obvious two first integrals: $I_0=H$ and $I_1=\lambda_5$. As described above, we can restrict this flow to the level set $\lambda_5=C$ and project this restriction onto the hyperplane $(x_1,x_2,y_1,y_2)$. We introduce the new variables:
 \begin{align*}
 p_1 = \lambda_1 + C y_1, \quad p_2 = \lambda_2 + C y_2, \quad p_3 = \lambda_3,\quad p_4 = \lambda_4,
 \end{align*}
 that simplify the form of the Hamiltonian function:
\begin{align*}
     H_C(q,p)=\frac{1}{2} \left(p_1^2 + p_2^2 + p_3^2 + p_4^2 - \frac{(p_3 x_1 + p_4 x_2 - p_1 y_1 - p_2 y_2)^2}{1 + x_1^2 + x_2^2 + y_1^2 + y_2^2}\right).
 \end{align*}
The Poisson structure on a $8$-dimensional symplectic manifold is not canonical any more:
\begin{align*}
  \{x_k,p_j\}&=\{y_k,p_{j+2}\}=\delta_{kj},\quad \{p_k,p_{j+2}\}=\delta_{kj}C,\\
   \{x_k, y_j\}&=\{x_k,p_{j+2}\}=\{y_k,p_j\}=\{p_j,p_k\}=\{p_{j+2},p_{k+2}\}=0,
\end{align*}
for $j, k=1,2$. Set $W=\frac{(p_3 x_1 + p_4 x_2 - p_1 y_1 - p_2 y_2)}{1 + x_1^2 + x_2^2 + y_1^2 + y_2^2}$ , the corresponding Hamilton's system is:
\begin{align*}
      \dot{x_1} &= \{x_1, H_C \} = p_1+y_1 W,&\dot{p_1} &= \{p_1, H_C \} =\left(C+W\right) (p_3-x_1 W),\\
      \dot{x_2} &= \{x_2, H_C \} = p_2+y_2 W,& \dot{p_2} &= \{p_2, H_C \} =\left(C+W\right) (p_4-x_2 W),\\
      \dot{y_1} &= \{y_1, H_C \} = p_3-x_1 W,& \dot{p_3} &= \{p_3, H_C \} = -\left(C+W\right) (p_1+y_1 W),\\
      \dot{y_2} &= \{y_2, H_C \} = p_4-x_2 W,&\dot{p_4} &= \{p_4, H_C \} = -\left(C+W\right) (p_2+y_2 W).
    \end{align*}
Let us now introduce the hyperspherical coordinates $(r,\theta_1, \theta_2, \theta_3)$:
 \begin{align*}
 x_1=r \cos\theta_1 \cos\theta_2, \quad y_1=r \cos\theta_1 \sin\theta_2,\quad
 x_2=r \sin\theta_1 \cos\theta_3, \quad y_2=r \sin\theta_1 \sin\theta_3.
 \end{align*}
 In the new coordinates, the Riemannian metric has the form:
 \begin{align}\label{eq:exR}
 dr^2 {+} {r^2} d\theta_1^2 {+} \frac{1}{2}r^2\left[\cos^2\theta_1(2{+}r^2{+}r^2\cos2\theta_1)d\theta_2^2{+} \sin^2\theta_1(2{+}r^2{-}r^2\cos2\theta_1)d\theta_3^2{+}r^2\sin^2 2\theta_1 d\theta_2 d\theta_3\right]
 \end{align}
 and the Hamiltonian function is:
 \begin{align*}
 \tilde H_C(r,\theta_1,\theta_2,\theta_3, p_r, p_{\theta_1}, p_{\theta_2}, p_{\theta_3})=\frac{1}{2}\left[p_r^2+\frac{1}{r^2} p_{\theta_1}^2+\frac{1}{r^2}\left(\frac{p_{\theta_2}^2}{\cos^2\theta_1}+\frac{p_{\theta_3}^2}{\sin^2\theta_1}\right)-\frac{1}{1+r^2}(p_{\theta_2}+p_{\theta_3})^2\right],
 \end{align*}
 with the following Poisson structure on the $8$-dimensional symplectic manifold:
    \begin{align*}
      \{r, p_r\} &=1,\quad \{r,  p_{\theta_k}\}=\{\theta_j, p_r\}= 0,\quad \{\theta_j, p_{\theta_k}\}=\delta_{jk},\quad j,k=1,2,3,\\
      \{p_r,p_{\theta_1}\}&=0, \quad\{p_r,p_{\theta_2}\}=  r C\cos^2\theta_1,\quad \{p_r,p_{\theta_3}\}=r C \sin^2\theta_1,\\
       \{p_{\theta_1},p_{\theta_2}\}&=-\{p_{\theta_1},p_{\theta_3}\}=-  r^2 C\sin\theta_1\cos\theta_1,\quad \{p_{\theta_2}, p_{\theta_3}\} = 0.
    \end{align*}

    The Hamilton equations are given by:
    \begin{align}
      \dot{r} &= \{r, \tilde H_C\} = p_r,\notag\\
       \dot{\theta}_1&=\{\theta_1, \tilde H_C\} = \frac{p_{\theta_1}}{r^2},\notag\\
       \dot{\theta}_2&=\{\theta_2, \tilde H_C\} = \frac{p_{\theta_2}}{r^2\cos^2\theta_1}-\frac{p_{\theta_2}+p_{\theta_3}}{1+r^2},\notag\\
       \dot{\theta}_3&=\{\theta_3, \tilde H_C\} = \frac{p_{\theta_3}}{r^2\sin^2\theta_1}-\frac{p_{\theta_2}+p_{\theta_3}}{1+r^2},\notag\\
       \dot{p}_r&=\{p_r, \tilde H_C\}=\frac{1}{r^3}p_{\theta_1}^2+\frac{p_{\theta_2}+p_{\theta_3}}{r(1+r^2)^2}\left((1+r^2)C-r^2(p_{\theta_2}+p_{\theta_3})\right)
       +\frac{1}{r^3}\left(\frac{p_{\theta_2}}{\cos^2\theta_1}+\frac{p_{\theta_3}}{\sin^2\theta_1}\right),\notag\\
       \dot{p}_{\theta_1}&=\{p_{\theta_1}, \tilde H_C\}= \frac{1}{r^2}\left(p_{\theta_3}\left(r^2C+\frac{p_{\theta_3}}{\sin^2\theta_1}\right)\cot\theta_1 -p_{\theta_2}\left(r^2C+\frac{p_{\theta_2}}{\cos^2\theta_1}\right)\tan\theta_1\right),\label{eq:ELt1}\\
       \dot{p}_{\theta_2}&=\{p_{\theta_2}, \tilde H_C\}=-C\cos\theta_1(r p_r\cos\theta_1-p_{\theta_1}\sin\theta_1),\label{eq:ELt2}\\
       \dot{p}_{\theta_3}&=\{p_{\theta_3}, \tilde H_C\}=-C\sin\theta_1(r p_r\sin\theta_1+p_{\theta_1}\cos\theta_1).\label{eq:ELt3}
    \end{align}
   The first two integrals are easily obtained from~\eqref{eq:ELt2} and~\eqref{eq:ELt3} since:
   \begin{align*}
     \frac{1}{2}\{r^2 C \cos^2\theta_1, \tilde H_C\}&=C\cos\theta_1(r p_r\cos\theta_1-p_{\theta_1}\sin\theta_1),\\
     \frac{1}{2}\{r^2 C \sin^2\theta_1, \tilde H_C\}&=C\sin\theta_1(r p_r\sin\theta_1+p_{\theta_1}\cos\theta_1).
   \end{align*}
   For the third one, we use~\eqref{eq:ELt1} and the fact that $\{p_{\theta_2}+p_{\theta_3}, \tilde H_C\}=-Cr p_r$ and get the following:
   \begin{align*}
     \{p_{\theta_1}^2, \tilde H_C\}&=2\{p_{\theta_1}, \tilde H_C\}p_{\theta_1}= \frac{1}{r^2}\left(p_{\theta_3}\left(r^2C+\frac{p_{\theta_3}}{\sin^2\theta_1}\right)\cot\theta_1 -p_{\theta_2}\left(r^2C+\frac{p_{\theta_2}}{\cos^2\theta_1}\right)\tan\theta_1\right)p_{\theta_1}\\
     &=-2p_{\theta_2}\left(-rC p_r+ p_{\theta_1}\left(C+\frac{p_{\theta_2}}{r^2\cos^2\theta_1}\right)\tan\theta_1\right)\\
     &+2p_{\theta_3}\left(rC p_r+p_{\theta_1}\left(C+\frac{p_{\theta_3}}{r^2\sin^2\theta_1}\right)\cot\theta_1\right)-2rC p_r(p_{\theta_3}+p_{\theta_2})\\
     &=-\{{p_{\theta_2}^2}{\cos^{-2}\theta_1},\tilde H_C\}-\{{p_{\theta_3}^2}{\sin^{-2}\theta_1},\tilde H_C\}+\{(p_{\theta_2}+p_{\theta_3})^2, \tilde H_C\}
   \end{align*}
    Therefore, the flow has four first integrals:
    \begin{align*}
      \tilde I_0& =\tilde H_C,\quad \tilde I_1 = p_{\theta_2} + \frac{r^2}{2} C \cos^2\theta_1,\quad \tilde I_2 = p_{\theta_3} + \frac{r^2}{2} C \sin^2\theta_1,\\
      \tilde I_3 &= p_{\theta_1}^2+{p_{\theta_2}^2}{\cos^{-2}\theta_1}+{p_{\theta_3}^2}{\sin^{-2}\theta_1}-(p_{\theta_2}+p_{\theta_3})^2.
    \end{align*}
    These integrals are involutive and functionally independent almost everywhere.
\end{example}

{Finally, we turn to the integrability of the LR system introduced in Proposition~\ref{prop:LRdef} for the standard metric. In this case, the Hamiltonian function~\eqref{eq:HamR} takes the form:
\begin{align}\label{eq:HamRcan}
H=\frac{1}{2}\left(\bar\lambda^\t (E{-}\frac{1}{1{+}\langle q, q\rangle}(Jq)(Jq)^\t) \bar\lambda{+}2\lambda_{2n{+}1}\bar\lambda^\t(u{+}\frac{\langle u, u\rangle}{1{+}\langle q, q\rangle}Jq){+}\lambda_{2n{+}1}^2(\langle u, u\rangle{-}\frac{\langle u, u\rangle^2}{1{+}\langle q, q\rangle})\right).
\end{align}
The following theorem is an immediate consequence of Theorem~\ref{thm:rightR}. For convenience, we introduce the following polynomials, which will be used to express the integrals in a more compact form:
\begin{gather}\label{eq:pol}
\begin{aligned}
    P_k&=y_{k+1}\lambda_1-x_1\lambda_{n+k+1}+x_{k+1}\lambda_{n+1}-y_{1}\lambda_{k+1},\\
    Q_k&= x_{1}\lambda_{k+1}-x_{k+1}\lambda_1-y_1\lambda_{n+k+1}+y_{k+1}\lambda_{n+1}+(x_1y_{k+1}-x_{k+1}y_1)\lambda_{2n+1},\\
    R_j&=\{P_1, P_j\}=\{Q_1, Q_j\}=x_2\lambda_{j+1}-x_{j+1}\lambda_2+y_2\lambda_{n+j+1}-y_{j+1}\lambda_{n+2},\\
    S_j&=\{P_1, Q_j\}=\{P_j, Q_1\}=x_2\lambda_{n+j+1}+x_{j+1}\lambda_{n+2}-y_2\lambda_{j+1}-y_{j+1}\lambda_2+(x_2x_{j+1}-y_2y_{j+1})\lambda_{2n+1},
\end{aligned}
\end{gather}
for $1\leq k<n$, $2\leq j<n$.
}
\begin{theorem}\label{thm:right}
The {sub-Riemannian} geodesic flow {corresponding to} the {standard} left-invariant Riemannian metric on $\Hn$ and the right-invariant distribution $\DD_R$, is a Hamiltonian system on $T^*\Hn$  with the Hamiltonian function~{\eqref{eq:HamRcan}}
integrable via first integrals
\begin{gather}
\begin{aligned}\label{eq:intRL}
  &I_0 = H, \\
  &I_j  = x_j\lambda_{n+j} - y_j\lambda_j +\frac{1}{2}(x_j^2-y_j^2)\lambda_{2n+1},\\
  &I_{n+k} = {P_k^2+Q_k^2},\\
  &I_{2n}= \lambda_{2n+1},\\
  &I_{2n+l}=\{I_{n+1}, I_{n+l+1}\}{=4(P_{l+1} (P_1 R_{l+1} - Q_1 S_{l+1}) + Q_{l+1} (P_1 S_{l+1} + Q_1 R_{l+1}))},
\end{aligned}
\end{gather}
{with polynomials $P_k$, $Q_k$, $R_{l+1}$ and $S_{l+1}$ defined by~\eqref{eq:pol}}, for $ j=1,\ldots,n$, $k=1,\ldots,n-1$, $l=1,\ldots,n-2$.
These integrals are functionally independent almost everywhere. If $n=1$ or $n=2$, the set $\mathcal{F}$ of the first integrals is commutative. For $n > 2$, the system is integrable in a non-commutative sense, {admitting $3n-1$ first integrals of degrees one, two, and three in the momenta}.
\end{theorem}

\begin{remark}
   For $j=1,\ldots,n$ set:
   \begin{align*}
       U_j=\frac{1}{2}\left(L_{g*}(\frac{\partial}{\partial x_j})+R_{g*}(\frac{\partial}{\partial x_j})\right),\  V_j=\frac{1}{2}\left((L_{g*}(\frac{\partial}{\partial y_j})+R_{g*}(\frac{\partial}{\partial y_j})\right),\  U_0=V_0=L_{g*}(\frac{\partial}{\partial z})=R_{g*}(\frac{\partial}{\partial z}),
   \end{align*}
   where $L_{g*}$ and $R_{g*}$ are given by~\eqref{eq:lfields} and~\eqref{eq:rfields}, respectively
   Then the linear integrals $I_{2n+1}$ and $I_j$ correspond to the Killing vector fields:
   \begin{align*}
       K_0=U_0=V_0,\quad K_j=x_j V_j-y_j U_j, \quad j=1,\ldots,n.
   \end{align*}
 Unlike the LL case where the Killing vector fields form a Lie algebra isomorphic to $\hn$, here the Killing vector fields commute. The isometry group is thus more restricted than in LL case.
\end{remark}

{For the standard metric, we have established integrability of the LR Hamiltonian system on $\Hn$ in the commutative sense for $n=1,2$ and in the non-commutative sense for $n>2$. This contrasts with the corresponding LL systems, which are integrable in non-commutative sense for every $n$. This sharp difference between the LL and LR cases suggests that the integrability properties of LR systems on two-step nilpotent Lie groups deserve further investigation.}

\begin{thank}
The authors would like to thank Prof. Bo\v{z}idar Jovanovi\'{c} for the introduction to this problem and for many useful discussions on this topic.
The authors are also grateful to the anonymous reviewers for their careful reading of the manuscript and for their particularly helpful comments and suggestions.
\end{thank}

\textbf{Data availability} No data are available for this paper.
\newline
\par\textbf{Declarations}
\par\textbf{Conflict of interest} None of the authors have a Conflict of interest to disclose.
\newline
\par\textbf{Funding}
This research  is partially supported by the Ministry of  Science, Technological Development and Innovation, Republic of Serbia, through the project 451-03-33/2026-03/200104.

\end{document}